\documentclass[11pt]{article}

\title{\sc\vskip-1.0em Quotients of Fourier algebras, and representations which are not completely bounded}
\author{Y. Choi and E. Samei}
\date{August 7, 2012}

\usepackage{sectsty}
\allsectionsfont{\sc}

\usepackage[usenames]{color}

\usepackage{amsmath,amsfonts,amssymb,amsthm}
\usepackage{amscd} 

\newcounter{pulse}
\numberwithin{pulse}{section}

\newcounter{thnum}

\newcounter{qnum}

\newtheorem{thm}[thnum]{Theorem}
\newtheorem*{fact}{Fact}
\newtheorem{lem}[pulse]{Lemma}
\newtheorem{cor}[pulse]{Corollary}
\newtheorem*{qustar}{Question}
\newtheorem{qu}[qnum]{Question}

\theoremstyle{definition}
\newtheorem{dfn}[pulse]{Definition}
\newtheorem*{remstar}{Remark}

\newcommand{\FA}{\operatorname{A}}

\newcommand{\Cst}{{\rm C}^*}
\newcommand{\VN}{\operatorname{VN}}
\newcommand{\SIN}{{\rm SIN}}

\newcommand{\Bdd}{{\mathcal B}}
\newcommand{\Cpct}{{\mathcal K}}
\newcommand{\cH}{{\mathcal H}}

\newcommand{\Cplx}{{\mathbb C}}
\newcommand{\bbF}{{\mathbb F}}
\newcommand{\fA}{{\mathfrak A}}

\newcommand{\dt}[1]{\textcolor{Maroon}{\sf#1}}
\newcommand{\norm}[1]{\left\Vert#1\right\Vert}
\newcommand{\abs}[1]{\left\vert#1\right\vert}

\newcommand{\starhm}{$*$-homo\-morphism}
\newcommand{\starrep}{$*$-repres\-entation}
\newcommand{\oss}{o.s.~structure}

\newcommand{\MAX}[1]{\operatorname{\sf max}#1}
\newcommand{\MIN}[1]{\operatorname{\sf min}#1}

\newcommand{\cbnorm}[1]{\norm{#1}_{\rm cb}}

\newcommand{\wstar}{\ensuremath{{\rm w}^*}}

\newcommand{\tili}{\widetilde{\imath}}
\newcommand{\lm}{\lambda}

\newcommand{\gumpf}{\bigskip\noindent{2010 MSC: 43A30 (primary); 46L07 (secondary)}}

\begin{document}
\maketitle

\begin{abstract}
We observe that for a large class of non-amenable groups $G$, one can find bounded representations of $\FA(G)$ on Hilbert space which are not completely bounded. We also consider restriction algebras obtained from $\FA(G)$, equipped with the natural operator space structure, and ask whether such algebras can be completely isomorphic to operator algebras; partial results are obtained, using a modified notion of Helson set which takes account of operator space structure. In particular, we show that when $G$ is virtually abelian and $E$ is a closed subset, the restriction algebra $\FA_G(E)$ is completely isomorphic to an operator algebra if and only if $E$ is finite.

\gumpf
\end{abstract}

\begin{section}{Representations of the Fourier algebra}\label{s:non-cb_rep}
The Fourier algebra $\FA(G)$ of a locally compact group $G$ was introduced by P.~Eymard~\cite{Eym_BSMF64}. As a Banach space, $\FA(G)$ can be identified with the unique isometric predual of the group von Neumann algebra of $G$; it therefore inherits a canonical operator space structure.
Equipped with this operator space structure, $\FA(G)$ can be regarded as a kind of ``dual object'' of the group algebra $L^1(G)$; in particular, when $G$ is abelian then $\FA(G)$ is canonically isomorphic (both as a Banach algebra and as an operator space) to $L^1(\widehat{G})$.

Motivated by classical results concerning bounded representations of $L^1(G)$ on Hilbert space, several authors have considered the following question:

\begin{qustar}
Let $G$ be a locally compact group, let $\cH$ be a Hilbert space, and let $\phi: \FA(G) \to \Bdd(\cH)$ be a completely bounded homomorphism. Is $\phi$ similar to a \starrep? That is, does there exist an invertible operator $T\in\Bdd(\cH)$ such that $T\phi(f)T^{-1}$ is self-adjoint for every real-valued $f\in\FA(G)$?
\end{qustar}

By recent work of M. Brannan and the second author \cite{BranSam_corep}, the answer is positive if $G$ is a SIN group; in particular, if it is either abelian, compact, or discrete. It is natural to ask if (for such $G$) one really requires complete boundedness of the homomorphism $\phi$; might it be the case that every bounded homomorphism from $\FA(G)$ in to $\Bdd(\cH)$ is completely bounded?
The starting point of our article is the following result, which answers this question in the negative.

\begin{thm}\label{t:non-cb}
Let $G$ be a locally compact group which contains a closed, discrete, nonabelian free subgroup~$F$. Then there exists a Hilbert space $\cH$ and a bounded homomorphism $\phi: \FA(G) \to \Bdd(\cH)$ that is not completely bounded. In particular, $\phi$ is not similar to a \starrep.
\end{thm}

Theorem~\ref{t:non-cb} follows by combining the results of \cite{BranSam_corep} with some known facts from the 1970s, and some basic facts about operator spaces and completely bounded maps (as may be found in~\cite{ER_OSbook}).
 The first main ingredient is due to M. Leinert \cite{Leinert_set}.

\begin{fact}[{\cite[equation (2.2)]{Leinert_set}}]
Let $F$ be a nonabelian free group. There exists an infinite set $E\subset F$ and a constant $C>0$ such that
\begin{equation}\label{eq:lein-ineq}
\norm{\chi_E f}_2 \leq C\norm{f}_{\FA(F)} \quad\text{for all $f\in \FA(F)$}
\end{equation}
\end{fact}

(Although we do not need estimates of the optimal constant $C$, we note that it was shown in \cite{Boz_PAMS75} that one can take $C=2$ and that this is sharp.)

\begin{remstar}
A few words may be in order concerning terminology. Leinert originally proved that sets $E$ which satisfied a certain combinatorial condition have the property \eqref{eq:lein-ineq}. In \cite{Boz_PAMS75}, this condition is referred to as \dt{Leinert's condition}; and sets inside locally compact groups that satisfy Leinert's condition are now known as \dt{Leinert sets}. (The term ``Leinert set'' has also been used for any set $E$ for which an inequality such as \eqref{eq:lein-ineq} holds.)
\end{remstar}

The second main ingredient we need is of similar vintage, and is originally due to A. M. Davie (a more general version was obtained independently
in~\cite{Var72_Q-alg}).

\begin{fact}[{\cite[Theorems 1.1 and 4.2]{Davie_Q-alg}}]
Let $1\leq p \leq 2$, and consider $\ell_p$ as a Banach algebra with pointwise product (and the usual norm). Then $\ell_p$ is isomorphic as a Banach algebra to a closed subalgebra of $\Bdd(\cH)$, for some Hilbert space~$\cH$.
\end{fact}

\begin{proof}[Proof of Theorem~\ref{t:non-cb}]
Let $F$ be as in the statement of the theorem. The main idea is that by the results of \cite{BranSam_corep}, every completely bounded homomorphism from $\FA(F)$ to $\Bdd(\cH)$ extends to a bounded homomorphism $c_0(F)\to\Bdd(\cH)$. We shall now construct a bounded homomorphism $\FA(F)\to\Bdd(\cH)$ that does not extend in this fashion.

For any subset $E\subseteq F$, we have $\ell_2(E)\subseteq \ell_2(F)\subseteq \FA(F)$ (since $F$ is a discrete group). Using Leinert's result, there exists an infinite subset $E$ and a constant $C>0$ for which the inequality \eqref{eq:lein-ineq} holds. Hence the map $c_{00}(F)\to \ell_2(E)$ defined by $f\mapsto \chi_E f$ extends to a surjective homomorphism of Banach algebras $q_E: \FA(F)\to\ell_2(E)$, with $\norm{q_E}\leq C$.

By Davie's result there is a Hilbert space $H$, a constant $c>0$, and a bounded homomorphism of Banach algebras $\theta: \ell_2(E) \to \Bdd(\cH)$ which satisfies
\[ \norm{\theta(a)}\geq c \norm{a}_2 \qquad\text{for all $a\in \ell_2(E)$.} \]
So $\theta q_E:\FA(F)\to\Bdd(\cH)$ is a bounded homomorphism of Banach algebras. Suppose it is completely bounded:
then by \cite[Theorem 20]{BranSam_corep} there exists $K>0$ such that
\[ \norm{\theta q_E(f)} \leq K\norm{f}_\infty \quad\text{for all $f\in \FA(F)$.} \]
In particular, $c\norm{a}_2 < K\norm{a}_\infty$ for all $a\in \ell_2(E)$; but this is absurd, since $E$ is an infinite set. Hence $\theta q_E$ is not completely bounded.

Now, let $q_F: \FA(G)\to\FA(F)$ be the contractive algebra homomorphism defined by $q_F(h)=h\vert_F$\/. Suppose the Banach algebra homomorphism
\[ \phi=\theta q_E q_F : \FA(G) \to \Bdd(\cH) \]
is completely bounded. By a well-known result of Herz, $q_F$ is surjective and the pre-adjoint of an injective \starhm\ (see \cite[Theorem~1]{Herz_AIF73}, or \cite[Theorem 3]{TakTat2}). Therefore $q_F$ is a complete quotient map, which implies that $\theta q_E$ is completely bounded, contrary to what was proved above.

We have shown $\phi$ is not completely bounded. It only remains to note that if $\phi$ were similar to a \starrep, then it \emph{would} be completely bounded; this follows from the observation that any \starrep\ $\FA(G)\to\Bdd(\cH)$ extends continuously to a \starhm\ $C_0(G)\to\Bdd(\cH)$, and hence is automatically completely bounded. (For details, see the proof of \cite[Theorem 8]{BranSam_corep}.)

\end{proof}

\begin{remstar}
It is possible to prove Theorem~\ref{t:non-cb} without using the results of \cite{BranSam_corep}, if one uses some more advanced results on various Hilbertian operator spaces.
This provides an alternative point of view which may be of interest to some readers; details are provided in a short appendix to the present paper.
\end{remstar}

\end{section}

\begin{section}{Restriction algebras as operator algebras?}
\label{s:cb-Helson}
By an \dt{operator algebra}, we mean a norm-closed subalgebra of $\Bdd(\cH)$ for some Hilbert space~$\cH$.
The previous section hinged on the fact that certain restriction algebras arising from the Fourier algebra are isomorphic as Banach algebras to operator algebras. It is then natural to ask what happens when operator space structures (henceforth abbreviated to {\oss}s) are taken into account; can we obtain examples where restriction algebras and operator algebras are not just isomorphic, but are completely boundedly isomorphic?

To be precise: if $E$ is a closed subset of a locally compact group $G$, we write $I_E$ for the ideal of functions in $\FA(G)$ that vanish on $E$, and define $\FA_G(E)$ to be the quotient algebra $\FA(G)/I_E$, equipped with the quotient norm and the quotient \oss; this may be thought of the algebra obtained from $\FA(G)$ by ``restriction to~$E$.''
Our question is then the following.

\begin{qustar}
Are there infinite sets $E\subseteq G$ for which $\FA_G(E)$ is completely boundedly isomorphic to an operator algebra?
\end{qustar}

We shall present some partial results which suggest that this question has a negative answer. 
First, let us fix some terminology and notation; details and definitions can be found in the standard references \cite{ER_OSbook,Pis_OSbook}.

The minimal and maximal {\oss}s on a given Banach space $V$ will be denoted by $\MIN{V}$ and $\MAX{V}$ respectively. An operator space $X$ is said to be \dt{minimal} if $X=\MIN{X}$, or \dt{maximal} if $X=\MAX(X)$.
 We shall use the abbreviations ``c.b.'' for ``completely bounded'', ``cb-isomorphic'' for ``completely boundedly isomorphic'', and so on.
If $X$ and $Y$ are operator spaces, by a \dt{cb-isomorphism from $X$ onto $Y$}, we mean a linear bijection $X\to Y$ which is completely bounded and whose inverse is also completely bounded.

$\FA(G)$ will always be equipped with the canonical \oss\ it inherits as the predual of $\VN(G)\subseteq \Bdd(L^2(G))$; and whenever we regard $\FA_G(E)$ as an operator space, it will always be with respect to the \oss\ induced by the quotient map $\FA(G)\to\FA_G(E)$.

\begin{remstar}
Suppose $E$ is contained in a closed subgroup of $G$, say $H$, so that we have two natural quotient maps to consider, namely $q_G:\FA(G) \to \FA_G(E)$ and $q_H: \FA(H) \to \FA_H(E)$. Consider the commutative diagram
\[ \begin{CD}
\FA(G) & @>{q_1}>> & \FA(H) \\
 @V{q_G}VV  & & @VV{q_H}V \\
\FA_G(E) & @>>> & \FA_H(E)
\end{CD} \]
where $q_1$ is the natural restriction homomorphism.
As remarked in the proof of Theorem~\ref{t:non-cb}), $q_1$ is a \emph{complete} quotient map; and $q_G$ and $q_H$ are complete quotient maps by definition of the {\oss}s on $\FA_G(G)$ and $\FA_H(E)$. Hence the bottom horizontal arrow is also a complete quotient map, and a little thought shows it is injective. Thus the natural homomorphism $\FA_G(E)\to \FA_H(E)$ is a completely isometric isomorphism.
\end{remstar}

It will also be convenient to introduce some new terminology. First, note that since $\FA(G)$ is a regular algebra of continuous functions on $G$, closed under conjugation, the canonical inclusion map $\imath_E:\FA_G(E)\to C_0(E)$ has dense range (by the Stone-Weierstrass theorem). Moreover, this map is injective, and is completely contractive when $C_0(E)$ is equipped with its minimal \oss.

Recall that a closed subset $E\subseteq G$ is said to be a \dt{Helson set} (in~$G$) if $\imath_E: \FA_G(E)\to C_0(E)$ is surjective, with $\norm{\imath_E^{-1}}$ being the so-called \dt{Helson constant} of~$E$.
 
\begin{dfn}
Let $E\subseteq G$ be a closed subset. We say that $E$ is a \dt{cb-Helson set} if $\imath_E$ is surjective and $\imath_E^{-1}: \MIN{C_0(E)} \to \FA_G(E)$ is completely bounded. The c.b.~norm of $\imath_E^{-1}$ is called the \dt{cb-Helson constant} of~$E$.
\end{dfn}

\begin{lem}\label{l:hereditary}
Left translates of cb-Helson sets are cb-Helson, with the same constant; and closed subsets of cb-Helson sets are cb-Helson, with constant at most that of the parent set.
\end{lem}

\begin{proof}
For each fixed $x\in G$, left translation by $x$ is a complete isometry of the group von Neumann algebra, and hence induces a complete isometry of $\FA(G)$. This proves the first claim.

To prove the second claim: let $E\subseteq G$ be a cb-Helson set and let $F\subseteq E$ be a closed subset. We have a commuting diagram of complete contractions
\[
\begin{CD}
\FA_G(E) @>{\imath_E}>> \MIN{C_0(E)} \\
@V{f}VV @VV{g}V \\
\FA_G(F) @>{\imath_F}>> \MIN{C_0(F)}
\end{CD}
\]
where the vertical arrows are given by restriction of functions. Now $f$ is a complete quotient map, by definition of the {\oss}s on $\FA_G(E)$ and $\FA_G(F)$; and $g$ is a complete quotient map, since (by a variation of the Tietze extension theorem)
 it is a surjective \starhm\ between $\Cst$-algebras. Since $\imath_E$ is surjective and $\imath_E^{-1}$ is c.b, an easy diagram chase shows $\imath_F$ is also surjective and $\cbnorm{\imath_F^{-1}}\leq \cbnorm{\imath_E^{-1}}$.
\end{proof}

\begin{thm}\label{t:characterization}
Let $G$ be a \SIN\ group and let $E$ be a closed subset. Then $\FA_G(E)$ is cb-isomorphic to an operator algebra if and only if $E$ is cb-Helson. (In particular, if $E$ is not a Helson set, $\FA_G(E)$ cannot be cb-isomorphic to an operator algebra.)
\end{thm}

\begin{proof}
One direction is trivial: if $E$ is cb-Helson then $\FA_G(E)$ is cb-isomorphic, as a completely contractive Banach algebra, to the operator algebra $\MIN{C_0(E)}$.

The other direction is similar to the proof of Theorem~\ref{t:non-cb}.
By assumption, there is some closed subalgebra $\fA\subset\Bdd(\cH)$ and a cb-isomorphism of algebras $\theta:\FA_G(E)\to\fA$. Writing $q_E$ for the complete quotient map $\FA(G)\to\FA_G(E)$, the homomorphism $\theta q_E: \FA(G)\to\Bdd(\cH)$ is~c.b. Hence, by 
\cite[Theorem 20]{BranSam_corep}, there exists $K>0$ such that
$\norm{\theta q_E(f)} \leq K\norm{f}_\infty$ for all $f\in \FA(G)$.
Since $\theta^{-1}$ is (completely) bounded,
\begin{equation}\label{eq:unif}
\norm{q_E(f)}\leq K\norm{\theta^{-1}}\norm{f}_\infty
\end{equation}
for all $f\in\FA(G)$. Moreover, since the restriction map $C_0(G)\to C_0(E)$ is a quotient map of Banach spaces, the inequality \eqref{eq:unif} remains valid if we replace the supremum norm over $G$ with the supremum norm over $E$.
Therefore $\imath_E$ has closed range and so is surjective.

As $\theta\imath_E^{-1}:C_0(E)\to\Bdd(\cH)$ is a bounded homomorphism, by well-known results of Dixmier and Day it must be similar to a $*$-homomorphism (see~e.g.~\cite[Theorem 9.7]{Pau_CBbook2}), and so is c.b.~as a map $\MIN{C_0(E)}\to\Bdd(\cH)$.
Thus $\imath_E^{-1} = \theta^{-1}\circ \theta\imath_E^{-1}:\MIN{C_0(E)}\to \FA_G(E)$ is the composition of two c.b.~maps, and is therefore c.b.~as claimed.
\end{proof}

In contrast to the much-studied notion of a Helson set, the cb-version seems to have gone unexplored. In particular, the authors do not know of any examples of infinite cb-Helson sets; and in Theorem~\ref{t:no-cb-Helson} we shall show that a large class of infinite Helson sets fail to be cb-Helson. Our proof goes via the following lemma.

\begin{lem}\label{l:core-case}
Let $G$ be a locally compact group and $H$ a closed, abelian subgroup. Let $F$ be a finite subset of $H$, which is cb-Helson with constant~$K$. Then $\abs{F}\leq 2K^2$.
\end{lem}

\begin{proof}
We start by identifying the \oss\ on $\FA_G(F)$.
As remarked near the start of this section, we may identify $\FA_G(F)$ completely isometrically with $\FA_H(F)=\FA(H)/I_F$. But since $\FA(H)=L^1(\widehat{H})$ is the predual of an abelian von Neumann algebra, it is a maximal operator space (see~\cite[\S3.3]{ER_OSbook}). Quotients of maximal operator spaces are maximal (\cite[Proposition~3.3]{Pis_OSbook}) and therefore $\FA_G(F)$ is a maximal operator space.

Now consider the identity map ${\rm id}$ from $\MIN{C_0(F)}$ to $\MAX{C_0(F)}$.
By assumption, the map $\imath_F^{-1}:\MIN{C_0(F)}\to \FA_G(F)$ has c.b.~norm~$K$.
Moreover, since $\imath_F$ is contractive as a map between Banach spaces and $\FA_G(F)$ is a \emph{maximal} operator space, the map $\imath_F: \FA_G(F)\to \MAX{C_0(F)}$ is \emph{completely} contractive. Thus by factorizing ${\rm id}: \MIN{C_0(F)}\to \MAX{C_0(F)}$ through $\FA_G(F)$, we see that it has c.b.~norm $\leq K$. On the other hand, by a result of Paulsen,
 the identity map $\MIN{\ell_\infty^n}\to\MAX{\ell_\infty^n}$ has c.b.~norm $\geq\sqrt{n/2}$: see \cite[Corollary~3.5 and Theorem~3.8]{Pis_OSbook} for details.
We conclude that $K \geq \sqrt{\abs{F}/2}$ as required.
\end{proof}

\begin{thm}\label{t:no-cb-Helson}
Let $G$ be a locally compact group and $E$ a cb-Helson subset. Suppose there exist closed abelian subgroups $H_1,\dots, H_m$ in $G$ and elements $y_1,\dots, y_m\in G$ such that $E\subseteq \bigcup_{i=1}^m y_iH_m$. Then $E$ is finite.
\end{thm}

\begin{remstar}
We do not know if there exist a locally compact group, and an infinite Helson subset in it, \emph{not} contained in some finite union of cosets of closed abelian subgroups.
\end{remstar}

\begin{proof}[Proof of Theorem~\ref{t:no-cb-Helson}]
Let $K$ be the cb-Helson constant of $E$, and fix a finite subset $F\subseteq E$. It suffices to show that $\abs{F}$ is bounded above by some constant depending only on $m$ and~$K$.

 For $i=1,\dots, m$ let $F_i=y_i^{-1}(F\cap y_iH) = y_i^{-1}F\cap H_i$; by Lemma~\ref{l:hereditary}, $F_i$ is cb-Helson, with constant $\leq K$.
Since $F_i$ is finite and contained in the closed abelian subgroup $H_i$, Lemma~\ref{l:core-case} tells us that $\abs{F_i} \leq 2K^2$. By construction, $F=\bigcup_{i=1}^m y_iF_i$ and so $\abs{F}\leq 2mK^2$, as required.
\end{proof}

Putting Theorems~\ref{t:characterization} and \ref{t:no-cb-Helson} together, we have the following corollary.

\begin{cor}
Let $G$ be a virtually abelian, locally compact group, and let $E$ be a closed subset of $G$. Then $\FA_G(E)$ is cb-isomorphic to an operator algebra if and only if $E$ is finite.
\end{cor}

\begin{proof}
Sufficiency is trivial so we only need to prove necessity. Thus, suppose $\FA_G(E)$ is cb-isomorphic to an operator algebra. Since $G$ is virtually abelian, it is \SIN, so Theorem~\ref{t:characterization} implies that $E$ must be a cb-Helson set. But then since $G$ is a union of finitely many cosets of a single abelian subgroup, Theorem~\ref{t:no-cb-Helson} implies $E$ is finite.
\end{proof}

\end{section}

\begin{section}{Further remarks}

\begin{subsection}{An explicit realization of $\ell_2$ as an operator algebra}\label{s:explicit}
For sake of brevity, our proof of Theorem~\ref{t:non-cb} used the fact that $\ell_p$ is isomorphic to an operator algebra for $1\leq p \leq 2$, without saying anything about the proof given in~\cite{Davie_Q-alg}. In fact, the results we quoted from \cite{Davie_Q-alg} establish something stronger; they show that $\ell_p$ is isomorphic to the quotient of some uniform algebra by some closed ideal. Algebras of this form are called \dt{Q-algebras}, and it is a result of B. J. Cole that every Q-algebra admits an isometric algebra embedding into $\Bdd(\cH)$ for some Hilbert space $\cH$; his argument is a variant on the GNS construction, and is very readably described in~\cite[\S50]{BonsDunc}.

One price of the generality of the arguments in \cite{Davie_Q-alg}, and their combination with Cole's theorem, is that one does not obtain an \emph{explicit} realization of $\ell_p$ as an operator algebra (or more precisely, a Hilbert space $\cH$ and an explicit set of mutually orthogonal idempotents in $\Bdd(\cH)$, whose closed linear span is isomorphic to $\ell_2$). 
The reader may therefore feel that our ``construction'' of a non-completely bounded homomorphism is, as it stands, not very constructive.

We now address this point, by describing a simple and direct algebra embedding of $\ell_2$ into $\Bdd(\cH)$, which may be of independent interest.
The construction is probably not new, but we did not find the embedding written down explicitly in the sources that we checked.

\paragraph{An explicit embedding.}
If $E$ is an index set, let $E_0$ be $E$ with an extra element $\omega$ adjoined, and for each $i\in E$ define $u_i \in \Cpct(\ell_2(E_0))$ by
\[ u_i (\delta_j) = \left\{ \begin{aligned} \delta_i & \quad\text{if $i=j$ or $j=\omega$} \\ 0 & \quad\text{ otherwise}
  \end{aligned} \right.
  \]
Then each $u_i$ is a rank-one projection in $\Cpct(\ell_2(E_0))$. Moreover, if $a\in c_{00}(E)$ then
\[ \norm{ \sum_{i\in E} a_i u_i } \geq \norm{ \sum_{i\in E} a_i u_i(\delta_\omega)}_2 = \norm{ \sum_{i\in E} a_i \delta_i}_2 = \norm{a}_2 \]
while for each $x\in \ell_2(E_0)$,
\[ \norm{\sum_{i\in E} a_i u_i(x)}_2 = \norm{ \sum_{i\in E} a_i (x_\omega +x_i) \delta_i }_2  \leq \norm{a}_2 \sup_{i\in E} \abs{x_\omega+x_i} \leq \norm{a}_2 \sqrt{2} \norm{x}_2\,.\]
Hence, the map $\theta: c_{00}(E) \to \Cpct(\ell_2(E_0))$ defined by $\theta(a) = \sum_i a_i u_i$ extends to a continuous homomorphism $\ell_2(E)\to\Cpct(\ell_2(E_0))$ that has closed range.
(A small variation on this construction allows one to embed $\ell_2$ as a closed subalgebra of $\prod_{n=1}^\infty M_n(\Cplx)$, at least when the index set $E$ is countable.)

\begin{remstar}
Cole's theorem does not characterize Q-algebras: that is, not every commutative, closed subalgebra of $\Bdd(\cH)$ is a Q-algebra. It is not clear if one can use the embedding $\theta$ to obtain any kind of ``concrete'' description of $\ell_2$ as a Q-algebra, so even for $p=2$ the results of \cite{Davie_Q-alg} seem to genuinely go further than our simple construction does.
\end{remstar}
\end{subsection}

\begin{subsection}{Hilbertian, completely contractive Banach algebras.}
Let us say that a completely contractive Banach algebra is \dt{Hilbertian} if its underlying Banach space is isomorphic (but not necessarily isometric) to a Hilbert space.
Leinert's result shows that we can obtain examples of such algebras as quotients of suitable Fourier algebras; Davie's result shows that we can also obtain such examples as nonself-adjoint operator algebras. The question of whether one could get both properties simultaneously was the original motivation for Theorem~\ref{t:characterization} (which answers the question in the negative).

It would be interesting to know more about the possible examples of Hilbertian, completely contractive Banach algebras. We leave this as a possible avenue for future work; note that examples of commutative, Hilbertian operator algebras are studied in some detail, in~\cite{BleMer_JOT}. 
\end{subsection}

\begin{subsection}{Open questions}
Using the estimate from \cite{Boz_PAMS75} and the construction in Section~\ref{s:explicit}, we see that the homomorphism constructed in Theorem~\ref{t:non-cb} can be chosen to have norm $\leq 2\sqrt{2}$. On the other hand, if $G$ is discrete and $\theta: \FA(G)\to \Bdd(\cH)$ is a \emph{contractive} algebra homomorphism, then for each $x\in G$ the idempotent $\theta(\delta_x)$ has norm $\leq 1$ and so must be self-adjoint; hence $\theta$ extends continuously to $C_0(G)$ and is therefore completely contractive. (We thank M. Brannan for this observation.) Therefore the following question seems natural.

\begin{qu}
Does there exist a constant $C>1$ such that, whenever $G$ is a discrete group, every algebra homomorphism $\FA(G)\to\Bdd(\cH)$ with norm $\leq C$ is automatically completely bounded?
\end{qu}

The groups to which Theorem~\ref{t:non-cb} applies are all non-amenable, motivating our next question.

\begin{qu}
Let $G$ be an amenable, locally compact group, and let $\cH$ be a Hilbert space. Is every bounded representation $\FA(G)\to\Bdd(\cH)$ automatically completely bounded? What if we assume, moreover, that $G$ is discrete?
\end{qu}

Our last question is implicit in the discussion of Section~\ref{s:cb-Helson}, but we restate it for emphasis.

\begin{qu}
Does there exist a locally compact group which contains an infinite cb-Helson set? Can the group be \SIN? Can it be compact?
\end{qu}

\end{subsection}

\subsection*{Acknowledgments}
The authors thank Michael Brannan and Matthew Daws for useful exchanges. Thanks are also due to David Blecher, for pointing out the reference~\cite{BleMer_JOT} in response to an earlier version of this article. YC was supported by NSERC Discovery Grant 402153-2011 and ES was supported by NSERC Discovery Grant 366066-2009.
\end{section}

\appendix

\section*{Appendix: an alternative proof of Theorem~\ref{t:non-cb}}
In this appendix, we show how Theorem~\ref{t:non-cb} may be proved without using the results of~\cite{BranSam_corep}, at the expense of using more advanced results from operator space theory.
The outline of this alternative argument is similar to the previous one: the differences are that we replace Leinert's theorem with an operator space version, due to Haagerup and Pisier \cite[\S1]{HaaPis_Duke}, and replace Davie's result with an operator space version, due to Blecher and Le Merdy \cite[\S2]{BleMer_JOT}.
The extra matrix-norm structure at our disposal then allows us to avoid using~\cite{BranSam_corep}.

We require some standard properties of three particular {\oss}s on $\ell^2$: these are $R\cap C$, its operator space dual $R+C$, and $\MAX{\ell^2}$.
A concise overview of these and some relatives is given in \cite[\S2]{BleMer_JOT}.
We write $\tili:\MAX{\ell^2} \to R+C$ for the canonical, completely contractive bijection.
Then ${\tili\ }^{-1}:R +C \to \MAX{\ell^2}$ is not completely bounded; this can be shown directly, but it also follows immediately from \cite[Theorem 2.1]{BleMer_JOT}.

Now $\MAX{\ell^2}$, equipped with pointwise multiplication, is cb-isomorphic to an operator algebra. More precisely:

\begin{fact}[Blecher--Le Merdy]
There is a Hilbert space $\cH$ and an algebra homomorphism $\theta:\ell^2 \to \Bdd(\cH)$, with closed range $\fA$, such that $\theta: \MAX{\ell^2(S)} \to \fA$ is a cb-isomorphism.
\end{fact}

This result follows from Blecher's characterization (up to c.b.\ isomorphism) of operator algebras (\cite[Theorem 2.2]{Ble_MathAnn95}, together with \cite[Theorem 2.1]{BleMer_JOT}. (See also \cite[Proposition~5.3.7]{BleMer_book}.)

\medskip

Let $\bbF_\infty$ be the free group of countably infinite rank, and let $S\subset\bbF_\infty$ be the canonical set of free generators. Define $\FA(S)$ to be the coimage of the restriction map $\FA(\bbF_\infty)\to c_0(S)$, equipped with pointwise multiplication and the quotient \oss.
We write $q_S:\FA(\bbF_\infty)\to \FA(S)$ for the resulting quotient map of operator spaces; this is also an algebra homomorphism. By duality, $\FA(S)^*$ may be identified (up to complete isometry) with $\VN(S)$, the \wstar-closed linear span of $S$ inside $\VN(\bbF_\infty)$, and $q_S^*$ may be identified with the inclusion of $\VN(S)$ into $\VN(\bbF_\infty)$.
Let $E_\lambda$ denote the norm-closed linear span of $S$ inside $\Cst_r(\bbF_\infty)$. We then have the following results:
\begin{fact}[Haagerup--Pisier]\
\begin{itemize}
\item There exists a cb-isomorphism $R\cap C \to E_\lm$. In particular, since $E_\lm$ is reflexive as a Banach space, it coincides with $\VN(S)$.
\item $E_\lm$ is completely complemented as an operator subspace of $\Cst_r(\bbF_\infty)$; the completely bounded projection map $\Cst_r(\bbF_\infty)\to E_\lm$ has a \wstar-\wstar-continuous extension to a completely bounded projection map $\VN(\bbF_\infty)\to \VN(S)$.
\item There exists a cb-isomorphism $\phi: \FA(S) \to R+C$.
\end{itemize}
\end{fact}

The first and second of these results can be extracted from the calculations in \cite[\S1]{HaaPis_Duke}; the third then follows by taking (pre-)adjoints. Proofs can also be found in \cite[\S9.7]{Pis_OSbook}.

\paragraph{\it Alternative proof of Theorem~\ref{t:non-cb}.}
Consider the map $h = \theta {\tili\ }^{-1} \phi : \FA(S) \to \fA$; this is a bounded algebra homomorphism, which cannot be completely bounded (since $\theta^{-1}$ and $\phi^{-1}$ are completely bounded, while ${\tili\ }^{-1}$ is not).
Now let $G$ be a locally compact group which contains a discrete, free, nonabelian subgroup. It is well known that this subgroup must, in turn, contain a copy of $\bbF_\infty$, and so (by Herz's restriction theorem) the restriction homo\-morphism $q_F:\FA(G) \to \FA(\bbF_\infty)$ is a quotient map of operator spaces. Hence $q_Sq_F: \FA(G) \to \FA(S)$ is a quotient map of operator spaces.

Therefore, as $h$ is not completely bounded, the bounded algebra homomorphism
\[ hq_Sq_F: \FA(G) \longrightarrow \fA \hookrightarrow \Bdd(\cH) \]
is also not completely bounded. From this, we can deduce -- using the same argument as in Section~\ref{s:non-cb_rep} -- that $hq_Sq_F$ is not similar to a $*$-representation.
\hfill$\Box$


\begin{thebibliography}{10}

\bibitem{BleMer_book}
{\sc D.~Blecher} and {\sc C.~Le~Merdy}, {\em Operator algebras and their modules---an
  operator space approach}, vol.~30 of London Mathematical Society Monographs.
  New Series, The Clarendon Press Oxford University Press, Oxford, 2004.
\newblock Oxford Science Publications.

\bibitem{Ble_MathAnn95}
{\sc D.~P. Blecher}, {\em A completely bounded characterization of operator
  algebras}, Math. Ann., 303 (1995), pp.~227--239.

\bibitem{BleMer_JOT}
{\sc D.~P. Blecher} and {\sc C.~Le~Merdy}, {\em Quotients of function algebras, and
  operator algebra structures on $\ell_p$}, J. Operator Theory, 34 (1995),
  pp.~315--346.

\bibitem{BonsDunc}
{\sc F.~F. Bonsall} and {\sc J.~Duncan}, {\em Complete normed algebras}, Ergebnisse
  der Mathematik und ihrer Grenzgebiete, Band 80, Springer-Verlag, New York,
  1973.

\bibitem{Boz_PAMS75}
{\sc M.~Bo{\.z}ejko}, {\em On {$\Lambda (p)$} sets with minimal constant in
  discrete noncommutative groups}, Proc. Amer. Math. Soc., 51 (1975),
  pp.~407--412.

\bibitem{BranSam_corep}
{\sc M.~Brannan} and {\sc E.~Samei}, {\em The similarity problem for {F}ourier
  algebras and corepresentations of group von {N}eumann algebras}, J. Funct.
  Anal., 259 (2010), pp.~2073--2097.

\bibitem{Davie_Q-alg}
{\sc A.~M. Davie}, {\em Quotient algebras of uniform algebras}, J. London Math.
  Soc. (2), 7 (1973), pp.~31--40.

\bibitem{ER_OSbook}
{\sc E.~G. Effros} and {\sc Z.-J. Ruan}, {\em Operator spaces}, vol.~23 of London
  Mathematical Society Monographs. New Series, The Clarendon Press Oxford
  University Press, New York, 2000.

\bibitem{Eym_BSMF64}
{\sc P.~Eymard}, {\em L'alg\`ebre de {F}ourier d'un groupe localement compact},
  Bull. Soc. Math. France, 92 (1964), pp.~181--236.

\bibitem{HaaPis_Duke}
{\sc U.~Haagerup} and {\sc G.~Pisier}, {\em Bounded linear maps between
  {$C^*$}-algebras}, Duke Math.~J., 71 (1993), pp.~889--925.

\bibitem{Herz_AIF73}
{\sc C.~Herz}, {\em Harmonic synthesis for subgroups}, Ann. Inst. Fourier
  (Grenoble), 23 (1973), pp.~91--123.

\bibitem{Leinert_set}
{\sc M.~Leinert}, {\em Faltungsoperatoren auf gewissen diskreten {G}ruppen},
  Studia Math., 52 (1974), pp.~149--158.

\bibitem{Pau_CBbook2}
{\sc V.~Paulsen}, {\em Completely bounded maps and operator algebras}, vol.~78
  of Cambridge Studies in Advanced Mathematics, Cambridge University Press,
  Cambridge, 2002.

\bibitem{Pis_OSbook}
{\sc G.~Pisier}, {\em Introduction to operator space theory}, vol.~294 of
  London Mathematical Society Lecture Note Series, Cambridge University Press,
  Cambridge, 2003.

\bibitem{TakTat2}
{\sc M.~Takesaki} and {\sc N.~Tatsuuma}, {\em Duality and subgroups. {II}}, J.
  Functional Analysis, 11 (1972), pp.~184--190.

\bibitem{Var72_Q-alg}
{\sc N.~T. Varopoulos}, {\em Some remarks on {$Q$}-algebras}, Ann. Inst.
  Fourier (Grenoble), 22 (1972), pp.~1--11.

\end{thebibliography}

\vfill
\begin{tabular}{l}
{\sl Address:}\\
Department of Mathematics and Statistics\\
McLean Hall\\
University of Saskatchewan\\
106 Wiggins Road, Saskatoon, SK\\
Saskatchewan, Canada S7N 5E6\\
{\sl Email:} {\tt choi@math.usask.ca} and {\tt samei@math.usask.ca}
\end{tabular}

\end{document}